\documentclass[12pt]{amsart}
\usepackage{amssymb,amsmath,amsthm,latexsym,mathtools,dsfont}
\usepackage{mathrsfs}
\usepackage[mathscr]{euscript}
\usepackage{mathtools}
\usepackage{mdwlist}
\usepackage{graphicx}
\usepackage{enumerate}
\usepackage[shortlabels]{enumitem}
\usepackage{a4wide}
\newcommand{\comment}[1]{}

\usepackage{physics}
\usepackage{color}
\definecolor{amber}{rgb}{1.0, 0.75, 0.0}

\newcommand{\nn}[1]{{\vert\kern-0.25ex\vert\kern-0.25ex\vert #1 
    \vert\kern-0.25ex\vert\kern-0.25ex\vert}}
\newcommand{\lnn}[1]{{\left\vert\kern-0.25ex\left\vert\kern-0.25ex\left\vert #1 
    \right\vert\kern-0.25ex\right\vert\kern-0.25ex\right\vert}}

\renewcommand{\leq}{\leqslant}
\renewcommand{\geq}{\geqslant}

\newcommand{\ou}{%
  \mathrel{%
    \vcenter{\offinterlineskip
      \ialign{##\cr$\forall$\cr\noalign{\kern-1.5pt}$\exists$\cr}%
    }%
  }%
}

\newtheorem{theorem}{Theorem}[section]
\newtheorem{lemma}[theorem]{Lemma}
\newtheorem{proposition}[theorem]{Proposition}

\newcounter{maintheorem}

\newtheorem{mainth}[maintheorem]{Theorem}

\theoremstyle{definition}

\theoremstyle{remark}

\newcounter{smallromans}

\newenvironment{romanenumerate}
{\begin{list}{{\normalfont\textrm{(\roman{smallromans})}}}%
  {\usecounter{smallromans}\setlength{\itemindent}{0cm}%
   \setlength{\leftmargin}{5.5ex}\setlength{\labelwidth}{5.5ex}%
   \setlength{\topsep}{.5ex}\setlength{\partopsep}{.5ex}%
   \setlength{\itemsep}{0.1ex}}}%
{\end{list}}

\newcounter{smallromansdash}

{\end{list}}

\newcounter{bigromans} 
  {\end{list}}

\makeatletter
\newcommand*{\itemequation}[3][]{%
  \item
  \begingroup
    \refstepcounter{equation}%
    \ifx\\#1\\%
    \else  
      \label{#1}%
    \fi
    \sbox0{#2}%
    \sbox2{$\displaystyle#3\m@th$}%
    \sbox4{\@eqnnum}%
    \dimen@=.5\dimexpr\linewidth-\wd2\relax
    \ifcase
        \ifdim\wd0>\dimen@
          \z@
        \else
          \ifdim\wd4>\dimen@
            \z@
          \else 
            \@ne
          \fi 
        \fi
      \@latex@warning{Equation is too large}%
    \fi
    \noindent   
    \rlap{\copy0}%
    \rlap{\hbox to \linewidth{\hfill\copy2\hfill}}%
    \hbox to \linewidth{\hfill\copy4}%
    \hspace{0pt}
  \endgroup
  \ignorespaces 
}
\makeatother  

\title[$(\lambda+)$-injective spaces need not be $\lambda$-injective]{A forgotten theorem of Pe{\l}czy\'nski: $(\lambda+)$-injective spaces need not be $\lambda$-injective---the case $\lambda\in (1,2]$}
\subjclass[2010]{Primary 46B04, 46B25 Secondary 46E15, 54G05}
\author[T.~Kania]{Tomasz Kania}
\address[T.~Kania]{Mathematical Institute\\Czech Academy of Sciences\\\v Zitn\'a 25 \\115 67 Praha 1\\Czech Republic  and  Institute of Mathematics and Computer Science\\ Jagiellonian University\\ {\L}ojasiewicza 6, 30-348 Krak\'{o}w, Poland
}
\email{kania@math.cas.cz, tomasz.marcin.kania@gmail.com}

\author[G.~Lewicki]{Grzegorz Lewicki}
\address[G.~Lewicki]{Jagiellonian University\\ {\L}ojasiewicza 6, 30-348 Krak\'{o}w, Poland
}
\email{Grzegorz.Lewicki@im.uj.edu.pl}

\thanks{The first-named author acknowledges with thanks support received from SONATA 15 No. 2019/35/D/ST1/01734.
}
\keywords{Injective Banach space, minimal projection}
\date{\today}

\begin{document}
\maketitle
\begin{abstract}
Isbell and Semadeni [Trans.~Amer.~Math.~Soc.~107 (1963)] proved that every infinite-dimensional $1$-injective Banach space contains a hyperplane that is $(2+\varepsilon)$-injective for every $\varepsilon > 0$, yet is is \emph{not} $2$-injective and remarked in a footnote that Pe{\l}czy\'nski had proved for every $\lambda > 1$ the existence of a $(\lambda + \varepsilon)$-injective space ($\varepsilon > 0$) that is not $\lambda$-injective. Unfortunately, no trace of the proof of Pe{\l}czy\'nski's result has been preserved. In the present paper, we establish the said theorem for $\lambda\in (1,2]$ by constructing an appropriate renorming of $\ell_\infty$. This contrasts (at least for real scalars) with the case $\lambda = 1$ for which Lindenstrauss [Mem.~Amer.~Math.~Soc.~48 (1964)] proved the contrary statement.
\end{abstract}
\section{Introduction}
Let $\lambda \geqslant 1$. A Banach space $X$ is $\lambda$-\emph{injective} whenever for any pair of Banach spaces $E\subseteq F$, every bounded linear operator $T\colon E\to X$ may be extended to an operator $\widehat{T}\colon F\to X$ with norm at most $\lambda\cdot \|T\|$. Since every Banach space $X$ embeds isometrically into $\ell_\infty(\Gamma)$ for some set $\Gamma$ and the latter space is $1$-injective, $\lambda$-injectivity is equivalent to the existence of a~(bounded, linear) projection from $\ell_\infty(\Gamma)$ onto (any isometric copy of) $X$ that has norm at most $\lambda$. We refer to \cite[Section 2.5]{Dales:2016} for a modern exposition of injective spaces.\smallskip

Let us say that a Banach space is $(\lambda+)$-\emph{injective}, whenever it is $(\lambda+\varepsilon)$-injective for all $\varepsilon > 0$. Lindenstrauss proved that every $(1+)$-injective real Banach space is $1$-injective (\cite[Theorem 6.10]{Lindenstrauss:1964}). (The case of complex scalars stubbornly remains open.) However, Isbell and Semadeni demonstrated that the statement of Lindenstrauss' theorem  with $1$ replaced by $2$ is false (\cite[Theorem 2]{IsbellSemadeni:1963}). They showed that the sought counterexample may always be arranged as a hyperplane in an arbitrary infinite-dimensional $1$-injective space. Interestingly, in a~footnote on p.~40 they announce the following result communicated to them by Pe{\l}czy\'nski: \emph{Let $\lambda > 1$. Then there exists a $(\lambda+)$-injective space that is not $\lambda$-injective.}\smallskip

Semadeni informed the first-named author in a personal communication in 2016 that neither has he got any recollection of the result nor of the circumstances in which Pe{\l}czy\'nski had communicated it to him. To the best of our knowledge, the result has not been rediscovered or reproved since. In the present paper we provide a proof in the restricted setting of $\lambda\in (1,2]$. 
\begin{mainth}\label{Th:A}
Let $\lambda \in (1,2]$. If $X$ is an infinite-dimensional 1-injective space, then $X$ has a $(\lambda+)$-injective renorming that is not $\lambda$-injective.\smallskip

More specifically, for every $\lambda \in (1,2]$ there exists a functional $f\in \ell_\infty^*$ of norm one such that $\ker f$ is $(\lambda+\varepsilon)$-complemented $(\varepsilon > 0)$, yet there is no minimal projection onto $\ker f$ of norm $\lambda$.
\end{mainth}

 The former part of Theorem~\ref{Th:A} follows from the latter one as every infinite-dimensional 1-injective space $X$ contains an isometric copy of $\ell_\infty$ (Lemma~\ref{lem:injective}), which is thus 1-complemented therein. Moreover, $\ell_\infty$ is isomorphic to its hyperplanes, hence so is $X$. Consequently, every infinite-dimensional 1-injective Banach space contains a hyperplane whose projection constant is $\lambda\in (1,2]$, yet there is no minimal projection thereonto.\smallskip

 The case of $\lambda = 2$ in Theorem~\ref{Th:A} subsumes the conclusion of \cite[Theorem~2]{IsbellSemadeni:1963} concerning 2-injectivity. Intriguingly, \emph{en route} to the proof we rely on a result by Blatter and Cheney \cite{BlatterChenney:1974} from 1974, that had been obviously unavailable to Pe{\l}czy\'nski in the 1960s, so his proof most likely must have been different. Furthermore, the complementary case of $\lambda \in (2,\infty)$ remains a mystery to us. \smallskip

\comment{
 The subspace of $\ell_\infty$ described in Theorem~\ref{Th:A} is constructed in terms of a certain singular functional. Let us now present the second main result that provides a general framework for choosing subspaces of $\ell_\infty$ of finite codimension for which not only can we prove the existence of a minimal projection, but also we can find its norm explicitly.\smallskip
 
 For this we require a piece of notation. Let us consider $X = \ell_\infty$. For given $n\geqslant 2$ and $k\in \mathbb N$ let us consider 
 \begin{itemize}
    \item $f\in (\ell_\infty^n)^*$ given by $\langle f, (x_j)_{j=1}^n\rangle = x_1 + \ldots + x_n$  $((x_j)_{j=1}^n\in \ell_\infty^n)$,
    \item $f^k \in (\ell_\infty(\ell_n))^*$ given by $\langle f^k, x\rangle = \langle f, \pi_k x\rangle$ $(x\in \ell_\infty(\ell_\infty^n))$, where $\pi_k\colon \ell_\infty(\ell_\infty^n)\to \ell_\infty^n$ is the $k$\textsuperscript{th} evaluation operator.
\end{itemize} 
Thus, in the above framework we identify $\ell_\infty$ with the $\ell_\infty$-sum of $\ell_\infty^n$ (with fixed $n$).
\begin{mainth}\label{Thm:B} 
Suppose that $X = \ell_\infty$. Let $n \geqslant 2$, $k\in \mathbb N$, $g\in c_0^\perp$, $\|g\|=1$. Set
\begin{itemize}
    \item $Y = \ker g \cap \bigcap_{k=1}^{\infty} \ker f^k$, $\widehat{Y} = \ker \widehat{g}  \cap \bigcap_{k=1}^{\infty} \ker \widehat{f^k} $,
    \item $Y_1 = \bigcap_{k=1}^{\infty} \ker f^k$, $\widehat{Y}_1 = \bigcap_{k=1}^{\infty} \ker(\widehat{f^k})$
\end{itemize}
and consider the restriction operator $ L\colon X^{**} \rightarrow X$, $Lx = x|_{\ell_1}$ $(x\in X^{**})$. Suppose that $(\widehat{g} - g \circ L)|_{\widehat{Y}_1}=0.$
Then
\begin{romanenumerate}
    \item $\lambda(\widehat{Y}, X^{**}) = \lambda(Y,X)$,
    \item $\lambda(Y,X) = 2-\frac{2}{n} + \frac{1}{\|g|_{Y_1}\|}$,
    \item there is a minimal projection in $\mathcal{P}(X,Y)$.
\end{romanenumerate}
\end{mainth}
}

\section{Preliminaries}
We use standard notation from Banach space theory. All Banach spaces are considered under the fixed field of real or complex numbers. A \emph{projection} is a bounded idempotent linear operator acting on a Banach space. We regard $\ell_\infty$ as the dual space of $\ell_1$ and similarly $\ell_\infty^n$ as $(\ell_1^n)^*$ ($n\in \mathbb N$). We shall require the fact that $\ell_\infty^*$ decomposes canonically into $\ell_1 \oplus_1 c_0^\perp$. Elements of $c_0^\perp$ are called \emph{singular} functionals and, of course, for every singular functional $f$ one has $\langle f, x\rangle = 0$ ($x\in c_0$). For a Banach space $X$ and $x\in X$ we denote by $\widehat{x} = \kappa_Xx\in X^{**}$, \emph{i.e.}, $\langle \widehat{x}, f \rangle = \langle f,x\rangle$ ($x\in X^{**}$). More generally, for a bounded linear operator $T\colon X\to Y$ we denote by $\widehat{T}\colon X^{**}\to Y^{**}$ the second adjoint of $T$.\smallskip

We record the following folk fact about norm-attainment of functionals in $c_0^\perp$. For the reader's convenience we include the proof. 

\begin{lemma}\label{lem1}
    Let $g \in c_0^{\perp},$ $ \|g\|=1.$ Then $g$ attains the norm if and only if there exists $ y=(y_i)_{i=1}^\infty \in \ell_{\infty}$ such that 
    \[
        \langle g, y \rangle =1 \hbox{ and } \limsup_{j\to \infty} |y_j| = 1.
    \]
    Moreover, since $\ell_\infty / c_0$ is non-reflexive, liftings of non-norm attaining functionals on $\ell_\infty/c_0$ to $\ell_\infty$  are functionals in $c_0^\perp$ that do not attain their norm.
\end{lemma}
\begin{proof}
If there is $ y \in X$ with $ \|y\|=1$ such that $\langle g, y\rangle = 1$, then 
\[
    1 = \|y\| \geq \limsup_{j\to \infty} |y_j|.
\]
Since $g \in c_0^{\perp}$, $\limsup_{j\to \infty} |y_j| =1$. Suppose that there exists $y \in X$ such that $ \langle g, y \rangle = 1$ and  $\limsup_{j\to \infty} |y_j| = 1$. Define $z = (z_j)_{j=1}^\infty \in X$ by $ z_j = {\rm sgn}(y_j)$, when $ |y_j|>1 $ and $z_j = y_j$ otherwise. We \emph{claim} that $ y-z \in c_0$. If not, there exists an increasing sequence of natural numbers $(j_k)_{k=1}^\infty$ and $d > 0$ such that 
\[
    \lim_{k\to\infty} |z_{j_k} - y_{j_k}| = d.
\]
Without loss of generality, we may assume that $ \{ j_k\colon k\in \mathbb N\} \subset \{ j \in \mathbb{N}\colon |y_j|>1 \}.$
Consequently
\[
    d = \lim_{k\to\infty} |y_{j_k} - {\rm sgn}(y_{j_k})|   = \lim_{k\to\infty} |y_{j_k}|- 1;
\]
a contradiction. Since $y-z \in c_0$ and $g \in c_0^{\perp}$ one has $\langle g, y \rangle = \langle g, z \rangle = 1$. As $\|z\|=1$, $g$ attains the norm at $z$.\end{proof}

For a Banach space $X$ and a closed subspace $Y$ of $X$ we denote by $\mathcal{P}(X,Y)$ the (possibly empty) set of projections from $X$ onto $Y$. In the case where this set is non-empty, we define the \emph{projection constant}
\[
    \lambda(Y,X) = \inf\{\|P\|\colon P\in \mathcal{P}(X,Y)\}.
\]
A projection $P\in \mathcal{P}(Y,X)$ is \emph{minimal}, whenever $\|P\| = \lambda(Y,X)$. For a future reference, we record the following simple lemma.

\begin{lemma}\label{isometries}
Let $X$ be a Banach space and suppose that $Y \subset X$ is a complemented subspace. If $T\colon X \rightarrow X$ is a linear surjective isometry, then 
\[
    \lambda(Y, X) = \lambda(T(Y), X).
\]
Moreover, there exists a minimal projection in $ \mathcal{P}(X,Y)$ if and only if there exists a minimal projection in $\mathcal{P}(X,T(Y)).$
\end{lemma} 

\begin{proof}
We define a transformation $\widehat{T}\colon \mathcal{P}(X,Y) \rightarrow  \mathcal{P}(X,T(Y))$ by $\widehat{T}(P) = T \circ P \circ T^{-1}$ ($P\in \mathcal{P}(X,Y)$). Then $ \widehat{T}$ is a surjective isometry from $ \mathcal{P}(X,Y)$ onto $\mathcal{P}(X,T(Y)),$ which proves the lemma. \end{proof}

Let us now specify to self-isometries of $\ell_\infty$.

\begin{lemma}
\label{positive}
Let $X =\ell_{\infty}$. Suppose that $f\in X^*$ is a norm-one functional decomposed as $ f = h + g$, where $ h = (h_i)_{i=1}^\infty \in \ell_1$ and $ g \in c_0^\perp$. Let $T\colon X \rightarrow X$ be defined as
\[
    Tx = (a_1x_1, a_2x_2, a_3x_3,\ldots)\quad (x\in X),
\]
where $a_i = \overline{{\rm sgn}(h_i)}$ if $h_i\neq 0$ and $a_i = 1$ otherwise $(i\in\mathbb N)$. Then $\lambda(Y ,X) = \lambda(T(Y),X)$, where $Y = \ker f$. Moreover, there exists a minimal projection in $ \mathcal{P}(X,Y)$ if and only if there exists a minimal projection in $ \mathcal{P}(X,T(Y)).$
\end{lemma} 
\begin{proof}
Observe that $ T$ is a linear, surjective isometry. The result follows from Lemma~\ref{isometries}.
\end{proof}
We shall also require the following consequence of \cite[Theorem 2]{BlatterChenney:1974}.
\begin{lemma}\label{lem:summing}
Let $h = (h_k)_{k=1}^\infty\in \ell_1$ be regarded a functional on $\ell_\infty$. If $h$ is non-zero and $\|h\|_{\ell_\infty} < 1/2$, $\|h\|_{\ell_1} = 1$, then 
\[
    \lambda(\ker h, \ell_\infty) = 1 + \Big( \sum_{k=1}^\infty \frac{|h_k|}{1 - 2 |h_k|}\Big)^{-1}.
\]
In particular, for $n\in \mathbb N$, $n\geqslant 3$ and $h^{\underline{n}} = (1/n, \ldots, 1/n) \in \ell_1^n$, we have
\[
    \lambda(Y_n, \ell_\infty) = 2 - \frac{2}{n},
\]
where $Y_{n} = \ker h^{\underline{n}}\subset \ell_\infty^n$ and the latter space is identified with a subspace of $\ell_\infty$.
\end{lemma}

1-injective Banach spaces are characterised as spaces isometric to the space $C(K)$ of continuous functions on a compact extremally disconnected space $K$ (\cite{Kelley}; see also \cite[Theorem~6.8.3]{Dales:2016}). (A topological space is \emph{extremally disconnected} whenever open subsets thereof have open closures.) We have been unable to locate the following apparently folk lemma in the literature so we provide its proof and we are indebted to K.P.~Hart for clarification regarding the proof. That every infinite-dimensional injective space contains an isomorphic copy of $\ell_\infty$ is well known and may be found, for example, in \cite[Proposition~1.15]{Aviles}.
\begin{lemma}\label{lem:injective}
    Every infinite-dimensional 1-injective Banach space contains an isometric copy of $\ell_\infty$.
\end{lemma}
\begin{proof}
    Let $K$ be an extremally disconnected space so that $X$ is isometric to $C(K)$. By the Balcar--Fran\v{e}k theorem \cite{BalcarFranek}, $K$ contains a copy of $\beta \mathbb N$ (since the minimal weight of an~infinite compact extremally disconnected space is continuum). Let us denote by $D$ the relatively discrete copy of the integers in the copy of $\beta\mathbb N$ in $K$. Thus, one may find disjoint open sets $(U_d)_{d\in D}$ in $K$ such that $U_d \cap D = \{d\}$ ($d\in D$). As $K$ is extremally disconnected, the closure of $U = \bigcup_{d\in D}U_d$ is a clopen subset of $K$ and $\overline{U} = \beta U$. The map $r\colon U\to D$ given by $r(x) = d$ if and only if $x\in U_d$ ($d\in D$) is a retraction. Consequently, it extends to a retraction $\beta r\colon \overline{U}\to \beta D$. Since $\overline{U}$ is clopen, we may extend $\beta r$ further to a retraction $p\colon K\to \beta D$ by declaring $p(x) = d_0$, where $d_0\in D$ is fixed and $x\in K\setminus \overline{U}$.
    
    The retraction $p$ yields a contractive linear projection $P\colon C(K)\to C(K)$ given by $Pf = f\circ p$ ($f\in C(K)$) with range isometric to $C(\beta \mathbb N)\equiv \ell_\infty$.
\end{proof}

\section{Proof of Theorem~\ref{Th:A}}
In order to prove Theorem~\ref{Th:A} we show that for any  $ \lambda \in (1,2]$ there is a subspace $W_\lambda$ of $ \ell_{\infty}$ 
such that $ \lambda(W_\lambda, \ell_\infty) =\lambda$ and for every $ P \in \mathcal{P}(\ell_{\infty}, W_\lambda)$ one has $\|P\| > \lambda$.\smallskip

We first consider the case of $\lambda \in (1,2)$ and the case $\lambda = 2$ will be treated separately. 

\begin{theorem}
\label{thm4}
    Let $X = \ell_{\infty}$. For every $\lambda \in (1,2)$ there exists a functional $f\in X^*$ of norm one such that $\lambda(\ker f, X) = \lambda$, yet there is no minimal projection onto $\ker f.$ 
\end{theorem}
\begin{proof}
First we consider the real case. Let $f \in X^*, $ $ \|f\| =1$ be such that $f = h + g$, where $g \in c_0^\perp$, $ h \in \ell_1$, $ 1 = \|h\| + \|g\|$, and $ 2|h_i| < \|h\|_{\infty}$ ($ i \in \mathbb{N}$). Set $Y = \ker f.$
By \cite[Theorem 2.4]{BF},  
\begin{equation}
\label{eq10}
    \lambda(Y,X) = 1 + \Big( \|g\| + \sum_{i=1}^{\infty} \frac{|h_i|}{ 1 - 2|h_i|}\Big)^{-1}.
\end{equation}
Let $ h^{\underline{n}} = (1/n, \ldots, 1/n) \in \ell_1^n$ and $Y_{n} = \ker h^{\underline{n}}.$ By Lemma~\ref{lem:summing}, $\lambda(Y_n, \ell_\infty^n) =  2 - \frac{2}{n}$. Fix $a\in [1/(n-1), 1]$ and let $Y_{a,n}$ be the kernel of
\[
    h_{a,n} = \Big(\frac{1}{1+(n-1)a}, \frac{a}{1+(n-1)a}, \ldots, \frac{a}{1+(n-1)a}\Big)\in \ell_1^n.
\]
By Lemma~\ref{lem:summing}, 
\[
    \lambda(Y_{a,n}, \ell_\infty^n) = 1 + \Big(\frac{1}{1+(n-1)a-2}+ \frac{(n-1)a}{1+(n-3)a}\Big)^{-1}.
\]
Fix $ g \in c_0^\perp,$ $ \|g\|=1$ such that $g$ does not attain its norm.
For any $ b \in [0,1]$ define
\begin{equation}
\label{ab}
    f_{n,a,b} = (1-b)g +bh_{a,n}.
\end{equation}
By \cite[Theorem 3.4]{BF}, there is no minimal projection onto the kernel of $f_{n,a,b}$.\smallskip

Let us fix $ \varepsilon >0$. We \emph{claim} that for any $\lambda \in (1+\varepsilon, 2 - \varepsilon)$ there is a subspace $W_\lambda$ of $\ell_\infty$ such that $\lambda(W_\lambda, \ell_\infty) = \lambda$ and there is no minimal projection onto $ W_\lambda$.
Fix $ n \in \mathbb{N} $ such that $ \lambda(Y_n) = 2 - 2/n > 2 - \varepsilon$ and observe that the function 
\[
    g(a)= 1 + \Big(\frac{1}{1+(n-1)a-2}+ \frac{(n-1)a}{1+(n-3)a}\Big)^{-1}
\]
is continuous, $ g(1) = 2 - 2/n$, and $ g(\frac{1}{n-1}) = 1$. Hence for $ b \in (0,1)$ close to $1,$  there is $ a \in [1/(n-1), 1]$ with $\lambda(\ker f_{n,a,b}, X) = \lambda$.\smallskip

Since $\varepsilon >0$ was arbitrary, for any $ \lambda \in (1,2)$ there exists $ W_\lambda$ such that $ \lambda(W_\lambda, X) = \lambda$ and there is no minimal projection onto $W_\lambda$, as required, so the proof in the real case is complete.\smallskip

The complex case requires an extra adjustment.
Reasoning as in \cite[Lemma 2.4]{BF}, we can show that for any $ P \in \mathcal{P}(X,Y)$ there exists $ y = (y_i)_{i=1}^\infty \in X$ such that 
$\langle f, y\rangle = 1$, $Px = x - \langle f, x\rangle y$ ($x \in X$), and  
\begin{equation}
\label{rew}
    \|P\| = \sup\{ |1-h_jy_j| + |y_j|(1-|h_j|)\colon j \in \mathbb{N}\}.
\end{equation}
Applying Lemmata~\ref{isometries}--\ref{positive}, we may assume that $ h_i \geq 0$ ($i \in \mathbb{N}$). 
Moreover, let $g \in c_0^\perp$ be a~norm-one functional that does not attain its norm (see Lemma~\ref{lem1} for details). Finally, let
\begin{itemize}
    \item $f_{n,a,b} = (1-b)g + bh_{a,n}$, and
    \item $P \in \mathcal{P}(X, \ker f_{n,a,b})$.
\end{itemize}
Then
\[
    Px = P_w x = x - \langle f_{n,a,b}, x\rangle w,
\]
where $ w \in X$ satisfies $\langle f_{n,a,b}, w\rangle = 1$. We set
\[
    y=({\rm Re}\, w_1, {\rm Re}\, w_2, \ldots, {\rm Re}\, w_n, w_{n+1},\ldots).
\]
Then $\langle f_{n,a,b}, y\rangle =1$. Moreover, it follows from \eqref{rew} that $\|P_w\| \geq \|P_y\|$. Reasoning as in \cite[Theorems 2.4 \& 3.4]{BF}, we arrive at the conclusion. 
\end{proof}

Let us now consider the case of $\lambda = 2$, which is parallel to \cite[Theorem~2]{IsbellSemadeni:1963}.

\begin{theorem}\label{thm1}
Let $g \in c_0^{\perp},$ $ \|g\|=1$ and let $Y = \ker g$. Then $\lambda(Y,X) =2$. Moreover, $Y$ is 2-injective if and only if $g$ attains its norm.
\end{theorem}
\begin{proof}
It is well-known that when $ W$ is a hyperplane in a Banach space $Z$ then $\lambda(W,Z) \leq 2$. 
Consequently, we need to show that $ \lambda(Y,X) \geq 2$. 

Let $P \in \mathcal{P}(X,Y)$. Then there exists $y \in X$ with $\langle g, y\rangle = 1$ such that $Px=x - \langle g, x\rangle y$ ($x \in X$). 
Since $\langle g, y\rangle = 1$ and $g \in c_0^{\perp}$, we have $\limsup_{j\to\infty} |y_j| \geq 1.$

Fix $\varepsilon > 0$ and $j \in \mathbb{N}$ such that $ |y_j| > 1-\varepsilon.$ Take $x=(x_j)_{j=1}^\infty \in X$ of norm one such that $ \langle g, x\rangle > 1-\varepsilon$. Since $ g \in c_0^{\perp}$, we may assume that $x_i = 0$ for $ i \leq j$. Let $w = (w_1,...,w_j) \in \ell_{\infty}^j$ be chosen so that $ \|w\|=1 $ and 
$w_j = - {\rm sgn}(\langle g, x\rangle)y_j$. Let $x^1 = w + x$. We note that $ \|x^1\| = 1$ and 
\[
    \|P\| \geq |(Px^1)_j| = | -{\rm sgn}(\langle g, x\rangle)y_j - \langle g, x\rangle y_j| = 1 + |\langle g, x\rangle||y_j| \geq 1 +(1-\varepsilon)(1-\varepsilon),
\]  
which shows that indeed $ \lambda(Y,X)=2$.\smallskip

\begin{itemize}
    \item If $g$ does not attain the norm and let $ P \in \mathcal{P}(X,Y)$ be as above. Since $g$ does not attain the norm, $\langle g, y\rangle > 1$, so by Lemma~\ref{lem1}, $\limsup_{j\to \infty}|y_j| = d >1.$
    Reasoning as above, we conclude that for any $ \varepsilon >0$ we have
        \[
            \|P\| \geq 1 +(1-\varepsilon)(d - \varepsilon).
        \]
    Consequently, $ \|P \| \geq 1 + d >2$, which proves the claim.
    \item If $g$ attains the norm, we may take $y \in X, $ with $\|y\|=1$ such that $\langle g, y \rangle =1$. Again, let $Px = x - \langle f, x\rangle y.$
    Obviously, $ P \in \mathcal{P}(X,Y)$ and 
        \[
            \|P\| \leq 1 + \|g\| \|y\| = 2.
        \]
    Hence $ \lambda(Y,X)=2 = \|P\|$, as required.
\end{itemize}\end{proof}

\comment{
\section{Proof of Theorem~\ref{Thm:B}}
\begin{lemma}\label{lem0}
    Let $X$ be an infinite-dimensional Banach space and let $(f_n)_{n=1}^\infty$ be a linearly independent sequence in $X^*$. Set $Y = \bigcap_{n=1}^\infty \ker f_n$. Suppose that 
\begin{itemize}
    \item $Y \subsetneq \bigcap_{n\neq j} \ker f_j$ ($j\in \mathbb N$).
    \item $Y$ is complemented in $X$ by a projection $P$.
\end{itemize}
Set $U = \ker P$. Then there exists a uniquely determined sequence $(y_n)_{n=1}^\infty$ such that
\begin{itemize}
\item $\langle f_j, y_n \rangle = \delta_{j,n}\quad (j,n\in \mathbb N)$
\item $W_n \subsetneq W_{n+1}$, where $W_n = {\rm span}\{y_i\colon i = 1, \ldots, n\}$ $(n\in \mathbb N)$.
\end{itemize} 
Moreover, setting
\begin{enumerate}
    \item $W = \bigcup_{n=1}^\infty W_n$,
    \item \label{eq:Zn}
        $Z_n = \bigcap_{j=n+1}^\infty \ker f_j$ $(n\in \mathbb N)$,
    \item\label{eq:Pn} $P_n \colon Z_n \to Y$ by the formula
        $P_n x = x - \sum_{j=1}^n \langle f_j, x\rangle y_j$ $(x\in X, n\in \mathbb N)$,
\end{enumerate}
for every $x\in Y\oplus W$ there is $n_x\in \mathbb N$ such that $P_nx = Px$ for $n\geqslant n_x$.\end{lemma}

\begin{proof}
Set $P_n = P|_{Z_n}$. As $\dim Y / Z_n = n$, there are 
$y_{1,n}, \ldots, y_{n,n}\in U$ with $\langle f_j, y_{i,n}\rangle = \delta_{i,j}$ ($i,j\in \mathbb N$) and 
\[
	P_n x = x - \sum_{j=1}^n \langle f_j, x\rangle y_{j,n}\quad (x\in X),
\]
We note that $P_m|Z_n = P_n$ for $m>n$. Consequently, for $i=1,\ldots, n$ we must have $y_{i,n} = y_{i,m}$. Setting $y_i = y_{i,i}$ we observe that $\langle f_j, y_i \rangle = \delta_{i,j}$. Certainly, this validates the proper inclusions  $W_n \subsetneq W_{n+1}$ for $n\in \mathbb N$. 


Let us take $ x \in Y \oplus W$ so that $x =y + w$, where $y \in Y$ and $w \in W.$ In particular, there exists $n_x \in \mathbb{N}$ such that $w \in W_n$ for $n \geq n_x$. It follows from the very definition of $ P_n $ that $ w \in \ker P_n$ for $ n \geq n_x$. Moreover, $P_n y = y $ for any $n \in \mathbb{N}$. Consequently, $ P_n x = y$ ($ n \geq n_x$) and $ P x =y $ as $w \in U$, which completes the proof.\end{proof}

\begin{lemma}\label{lemma:2}
Let $X = \ell_\infty$, $n\geqslant 2$, $k\in \mathbb N$, and $g\in c_0^\perp$, $\|g\|=1$. Moreover, let
\begin{itemize}
    \item $f\in (\ell_\infty^n)^*$ be given by $\langle f, (x_j)_{j=1}^n\rangle = x_1 + \ldots + x_n$  $((x_j)_{j=1}^n\in \ell_\infty^n)$,
    \item $f^k \in (\ell_\infty(\ell_n))^*$ be given by $\langle f^k, x\rangle = \langle f, \pi_k x\rangle$ $(x\in \ell_\infty(\ell_\infty^n))$, where $\pi_k\colon \ell_\infty(\ell_\infty^n)\to \ell_\infty^n$ is the $k$\textsuperscript{th} evaluation operator.
\end{itemize}. We set
\begin{itemize}
    \item $Y = \ker g \cap \bigcap_{k=1}^{\infty} \ker f^k$,
    \item $Y_1 = \bigcap_{k=1}^{\infty} \ker f^k,$
\end{itemize}
and define a projection $\widetilde{Q} \colon \ell_\infty^n \to \ell_\infty^n$ by the formula $\widetilde{Q}x = x - \langle f,x \rangle (1, \ldots, 1)$ $(x\in \ell_\infty^n)$. Identifying canonically $\ell_\infty$ with $\ell_\infty(\ell_\infty^n)$, we define a projection $Q\colon \ell_\infty\to Y_1$ by
\[
    Q(x^1,\ldots ,x^k, \ldots) =(\widetilde{Q}x^1,\ldots, \widetilde{Q}x^k,\ldots) \qquad (x^j\in \ell_\infty^n, j\in \mathbb N).
\]
Suppose that $y\in \ell_\infty$ satisfies $\langle g, y\rangle = 1$ and $\langle f^k, y\rangle = 0$ $(k\in \mathbb N)$ and set $Px = Qx -\langle g,x\rangle y$.
Then $ P \in \mathcal{P}(X, Y)$ and $\|P\| = \|\widetilde{Q}\| + \|y\|.$
\end{lemma}

\begin{proof}
For $k\in \mathbb{N}$ we define $y^{\underline{k}}\in \ell_\infty(\ell_\infty^n)$ to be $(1,\ldots, 1)\in \ell_\infty^n$ in the $k$\textsuperscript{th} coordinate and zero otherwise. By Lemma~\ref{lem0} applied to $P$ we infer that
\[
    U = \ker P = \overline{\bigcup_{n=1}^{\infty} W_n}, 
\]
where $W_n = {\rm span}\{y, y^{\underline{1}},\ldots,y^{\underline{n}}\}.$ 
Let $x\in X = \ell_\infty(\ell_\infty^n)$. Then for any $k \in \mathbb{N},$ 
\[
    \langle f_k, Px\rangle = \langle f_k, Qx \rangle - \langle g,x\rangle\langle f_k, y\rangle = 0.
\]
By Lemma~\ref{lem0}, for any $n \in \mathbb{N}$  and $x \in Y \oplus W_n,$ 
\[
    \langle g, Px \rangle = \left\langle g, x - \sum_{j=1}^n \langle f^j, x\rangle y^{\underline{j}}\right\rangle = 0
\]
For every $ x \in Y \oplus W$, we have $Px \in Y$. Since $Y$ is closed and $ Y \oplus W $ is dense in $X$, $Px \in Y$ for any $x \in X$. If $ x \in Y$, then 
$Px = Qx - \langle g,x\rangle y = x$. Consequently, $ P \in \mathcal{P}(X,Y)$.
Let $x \in X$ be a norm-one vector. Fix $ i \in \mathbb{N}$ and write $ i = kn +j,$ where $j \in \{1,\ldots ,n\}.$ We have
\[
    \|(Px)_i\| = \| \widetilde{Q}x^k - \langle g, x\rangle y_i\| \leqslant \|Q\|+ \|g\|\|y\|, 
\]
which shows that $\|P\| \leqslant \|\widetilde{Q}\| + \|y\| = n + \|y\|.$

For the converse inequality, first assume that $ \|y\| = \| y_i\|$ for some $ i \in \mathbb{N}.$ Write $i = kn +j$, where $j \in \{1,\ldots ,n\}.$
Set 
\[
    z= (-1,\ldots ,\underbrace{-1}_{j-1},\underbrace{1}_{j},\underbrace{-1}_{j+1},\ldots ,-1)\in \ell_\infty^n
\]
and observe that $ (\widetilde{Q}z)_j = \|\widetilde{Q}\| = n.$ For $x\in \ell_\infty^n$, let $x^{\underline{k}}$ denote the vector in $\ell_\infty(\ell_\infty^n)$ that is equal to $x$ on the $k$\textsuperscript{th} coordinate and zero otherwise. Fix $ \varepsilon >0$ and a norm-one vector $ v \in X$ such that 
\begin{itemize}
    \item $\langle g,v\rangle > \|g\|- \varepsilon = 1- \varepsilon$,
    \item $v_j = 0 $ for $ j \leqslant (k+1)n.$
\end{itemize} 
Such a choice is possible since $ g \in c_0^\perp$. Then for $ w = x+v$, we have $ \|w\| = 1$ and 
\[
    \|P\| \geqslant \|Pw\| \geqslant \|(Pw)_i\| = \|(\widetilde{Q}z)_j - \langle g,v\rangle y_i\|.
\]
Replacing $v$ by $-v$, if necessary, we observe that
\[
    \|(\widetilde{Q}z)_j - \langle g,v\rangle y_i\| \geqslant \|(\widetilde{Q}z)_j\| + (\|g\|-\varepsilon)\|y\| = \|\widetilde{Q}\| +(\|g\|-\varepsilon)\|y\|.
\]
We have thus showed that $ \|P\| = \|\widetilde{Q}\|+ \|y\|$ in this particular case.

In the complementary case where $ \|y\| > \|y_j\|$ ($j \in \mathbb{N}$), we have still $ \|y\| = \limsup_{j\to\infty} \|y_j\|.$ Let us fix $ \varepsilon >0$ and $ i \in \mathbb{N}$ with $ \|y_i\| > \|y\|- \varepsilon$; write $ i = kn + j$ with $j \in \{1, \ldots, n\}$. Reasoning as above, we arrive at the inequality 
\[
    \|P\| \geqslant  \|\widetilde{Q}\| +(\|g\|-\varepsilon)(\|y\| - \varepsilon).
\]
so $\|P \| \geqslant  \|\widetilde{Q}\|+ \|y\|,$ as required.\end{proof}

\begin{proposition} \label{thm6}
Fix $n\geqslant 2$. Let $ X, Y, Y_1, g, (f^k)_{k=1}^\infty$ be as in the statement of Lemma~\ref{lemma:2}. We consider the following subspaces of $X^{**}$:
\begin{itemize}
    \item $\widehat{Y}   = \ker \widehat{g}  \cap \bigcap_{k=1}^{\infty} \ker \widehat{f^k}$,
    \item $\widehat{Y}_1 = \bigcap_{k=1}^{\infty} \ker \widehat{f^k}$.
\end{itemize}
Then there is $\widetilde{Q} \in \mathcal{P} (X^{**}, \widehat{Y})$ such that $\|\widetilde{Q}\| = \lambda(\widehat{Y}, X^{**})$ and
\begin{equation}\label{eq:Qtilde}
    \widetilde{Q}x = x-\sum_{k=1}^{\infty}\langle \widehat{f^k}, x\rangle y^{\underline{k}} - \langle \widehat{g}, x\rangle y^g \quad (x \in X^{**}),\end{equation}
where 
\begin{itemize}
    \item $y^{\underline{k}}$ are defined as in the proof of Lemma~\ref{lemma:2},
    \item $y^g \in X^{**}$ satisfies $\langle \widehat{g}, y^g\rangle =1$ and  $\langle \widehat{f^k}, y^g\rangle = 0$ $(k\in\mathbb N)$.
\end{itemize}
\end{proposition}
\begin{proof}
Let 
\[
    P_1x = x-\sum_{k=1}^{\infty}\langle \widehat{f^k}, x\rangle y^{\underline{k}},
\]
where 
\begin{itemize}
    \item $y^{\underline{k}}$ ($k\in \mathbb N$) are defined as in the proof of Lemma~\ref{lemma:2}
    \item $ y^g \in X^{**}$ is chosen so that $ \langle \widehat{g}, y^g\rangle =1$ and $ \langle \widehat{f^k}, y^g \rangle= 0$ ($k \in \mathbb{N}$).
\end{itemize}  
We observe that the formula
\[
    P_2x = P_1x - \langle g, P_1x\rangle y^g\quad (x\in X^{**})
\]
defines a bounded projection from $X^{**}$ onto $ \widehat{Y}.$ Since $ \widehat{Y}$ is weakly-$^*$ closed in $X^{**},$ there exists a possibly different projection $ P \in \mathcal{P} (X^{**}, \widehat{Y})$ that is minimal.

Denote by $\sigma_n$ the set of all permutations of $ \{1,\ldots ,n\}.$ 
For $g \in \sigma_n$ and $ k \in \mathbb{N}$, let $g_k\colon X \rightarrow X $ be defined by 
\[
    g_k(x_j)_{j=1}^\infty = (x_1,\ldots ,x_{kn},x_{kn+g(1)},\ldots ,x_{kn+g(j)}, \ldots ,x_{kn+g(n)},x_{(k+1)n},\ldots ).
\]
for $(x_j)_{j=1}^\infty \in \ell_\infty(\ell_\infty^n)$. Every map $g_k\colon X\to X$ is a linear surjective isometry. We then set \begin{romanenumerate}
    \item $G_k = \{ g_k\colon g \in \sigma_n, k \in \mathbb{N}\}$,
    \item $H_k = G_1 \times G_2 \times \ldots  \times G_k $,
    \item \label{eq:Qp}$Q^k_{P} =\frac{1}{(n!)^k}\sum_{h \in H_k} (\widehat{h}^{-1}\circ P \circ \widehat{h}). $
\end{romanenumerate}

Let $ k \in \mathbb{N}$ and $h\in H_k$. Since $h$ is bijective, so is $\widehat{h}$. We \emph{claim} that
$ \widehat{h}(\widehat{Y}) \subset \widehat{Y}. $

To see this, take $ w \in \widehat{Y}.$ Then $\widehat{h}(w) = w \circ h^*,
$
where $ h^*$ is the adjoint operator to $h.$
By Goldstine's Theorem, there exists a net $ (z_\alpha) \subset X$ such that $ \widehat{z_\alpha} \rightarrow w$ weakly-$^*$ in $ X^{**}.$ 
Then for any $ k \in \mathbb{N},$ 
\begin{equation}
\label{imp}
\langle \widehat{h}(\widehat{z_\alpha}), f_k \rangle = \langle \widehat{z_\alpha}\circ h^*, f_k \rangle = \langle \widehat{z_\alpha}, f_k \circ h\rangle = \langle f_k, h(z_\alpha) \rangle  = \langle f_k, z_\alpha\rangle.
\end{equation}
Analogously, $ \langle \widehat{h}(\widehat{y_b}), g\rangle = \langle g, z_\alpha\rangle.$
For $ m \in \mathbb{N} \setminus \{0\}$ we set
\[
    V_{k,h,m} =\big\{ z \in X^{**}\colon \max\{ |\langle \widehat{h}(w-z), f^j\rangle|, |\langle \widehat{h}(w-z), g\rangle|\colon j=1,\ldots , k\} \leq 1/m\big\}.
\] 
Since $ \mathcal{V}  = \{V_{k,h,m}\colon k,m\in \mathbb N, h\in H_k\}$ is countable, we may enumerate it as $\{V_l\colon l \in \mathbb{N}\}.$
Since $ \widehat{z_\alpha} \rightarrow  w$ weakly-$^*$ in $ X^{**},$ we may find a sequence $ (\widehat{z_{l_j}})_{k=1}^\infty$ such that 
$ z_{l_j} \in V_{l_k}.$ We set
\[
    P_kz  = z - \sum_{j=1}^k \langle \widehat{f^j}, z\rangle y^{\underline{j}} \quad ( z \in X^{**}).
\] 
Since $ w \in \widehat{Y}$, we have $P_kw =w$. Moreover, $P_k z_{l_j} \in Y.$ Since $ \widehat{z_\alpha} \rightarrow w$ weakly-$^*$ in $X^{**},$
\begin{itemize}
    \item $\langle \widehat{f^j}, P_kz_{l_j} \rangle \rightarrow_j \langle \widehat{f^j}, P_k w\rangle = \langle \widehat{f^j}, w\rangle$,
    \item $\langle \widehat{g}, P_k z_{l_j} \rangle \rightarrow_j \langle \widehat{g}, P_k w\rangle = \langle \widehat{g}, w\rangle$.
\end{itemize}
By \eqref{imp}, for $ h \in H_k,$ $l \in \mathbb{N}$ and $j=1,\ldots, k$ we have
\begin{itemize}
    \item $\langle (\widehat{h}\widehat{P_k}) z_l, f_j\rangle = \langle f_j, h(P_kz_l)\rangle = \langle f_j, P_k z_l\rangle =0$,
    \item $\langle (\widehat{h}(\widehat{P_k}) z_l, g\rangle = \langle g, h P_k z_l)\rangle = \langle g, P_k z_l \rangle = 0$.
\end{itemize}
Consequently, $\widehat{h}w \in \widehat{Y}$ ($ h \in H_k$) so $ Q^k_P \in \mathcal{P} (X^{**}, \widehat{Y})$ and  $\|Q^k_P \| \leq \|P\| = \lambda(\widehat{Y}, X^{**})$. Let us consider the product space
\[
    \prod_{x \in X^{**}} (B_{X^{**}}\big(0, \lambda(\widehat{Y},X^{**}) + 1\big), w^*)
\]
that is compact. By compactness, the sequence $(Q^k_P)_{k=1}^\infty$ has a cluster point $ \widetilde{Q} \in \mathcal{P} (X^{**},\widehat{Y}).$
It is clear that $ \widetilde{Q} $ is a minimal projection. We \emph{claim} that $\widetilde{Q}$ is given by \eqref{eq:Qtilde}.  

By Lemma~\ref{lem0}, there exist a uniquely determined sequence 
$(y_n)_{n=1}^\infty$ in $X^{**}$ and $ y^g \in X^{**}$ such that $\langle \widehat{f^j}, y_i\rangle = \delta_{i,j},$ $ \langle \widehat{g}, y_j \rangle =0$, $ \langle \widehat{g}, y^g \rangle = 1$, and $ \langle \widehat{f^k}, y^g \rangle =0.$
We \emph{claim} that $ y_k = y^{\underline{k}}$ ($k \in \mathbb{N}$). For this, observe that for any $ l \in \mathbb{N}, $ $h \in H_l,$ we have $\widehat{h}(y^{\underline{k}})= h(y^{\underline{k}}) =y^{\underline{k}}$.

By \eqref{eq:Qp}, for $ l\geq k,$ $ j \in \{0,\ldots ,k\}$, and $ p \in \{1,\ldots ,n\},$
\[
    Q^l_P(y^{\underline{k}})_{(j-1)n+p} =Q^l_P(y^{\underline{k}})_{(j-1)n+1}.
\]
Since $ Q^l_P(y^{\underline{k}}) \in \widehat{Y}, $ $ \widetilde{Q}(y^{\underline{k}}) =0$ for any $ k \in \mathbb{N}.$
We note that $ \langle f^j, y^{\underline{i}}\rangle =\delta_{ij}$ and $ \langle g, y^{\underline{k}}\rangle =0$ ($k \in \mathbb{N}$). Hence $ \widetilde{Q}$ is determined by $ (y^{\underline{k}})_{k=1}^\infty$ together with $ y^g \in X^{**}$ such that 
$\langle g, y^g\rangle = 1 $ and $\langle f^k, y^g\rangle =0.$ For $p\in \mathbb N$ let $Q_p$ be defined as \eqref{eq:Qp}. By Lemma~\ref{lem0}, for $x \in \widehat{Y} \oplus W $ 
\[
    \lim_{p\to \infty}  Q_{p}x =  x-\sum_{k=1}^{n}\langle \widehat{f^k}, x\rangle y^{\underline{k}} - \langle \widehat{g}, x\rangle y^g = \widetilde{Q}x
\] 
for some $n_x\in \mathbb N$ and all $ n\geq n_x$. Moreover $ \|Q_p\| \leq \lambda(Z) + \|y^g\|\|g\|$ ($ p \in \mathbb{N}$),  
where $ Z$ is the kernel of the summing functional on $\ell_{\infty}^{n}.$ By the Banach--Steinhaus theorem, 
\[
    \widetilde{Q}(x)= x-\sum_{k=1}^{\infty}\langle \widehat{f^k}, x\rangle y^{\underline{k}} - \langle \widehat{g}, x\rangle y^g\quad (x \in X^{**})
\]
as required.
\end{proof}

\begin{proposition}\label{prop:B:final}
\label{thm7} Fix $n\geqslant 2$. Let $X, Y, Y_1, g, (f^k)_{k=1}^\infty$, and $\widetilde{Q}$  be as in Lemma~\ref{thm6}. We consider the following subspaces of $X^{**}$:
\begin{itemize}
    \item $\widehat{Y} = \ker \widehat{g}  \cap \bigcap_{k=1}^{\infty} \ker \widehat{f^k} $,
    \item $\widehat{Y}_1 = \bigcap_{k=1}^{\infty} \ker(\widehat{f^k})$
\end{itemize}
Let $ L\colon X^{**} \rightarrow X$ the operator given by $Lx = x|_{\ell_1}.$ Suppose that $(\widehat{g} - g \circ L)|_{\widehat{Y}_1}=0.$
Then
\begin{romanenumerate}
    \item $\lambda(\widehat{Y}, X^{**}) = \lambda(Y,X)$,
    \item $\lambda(Y,X) = 2-\frac{2}{n} + \frac{1}{\|g|_{Y_1}\|}$,\\ where 
    \item $ \widetilde{U} = L \circ \widetilde{Q}|_{X}$ is a minimal projection in $ \mathcal{P}(X,Y)$. 
\end{romanenumerate}

\end{proposition}
\begin{proof}
Let us note that $L \circ \widetilde{Q}|_{X} \in \mathcal{P}(X, Y).$ 
Fix $ x \in X$ and observe that 
\[
\begin{aligned}
\langle f^j, L\circ \widetilde{Q}x\rangle & = \langle f^j, L(x-\sum_{k=1}^{\infty} \langle \widehat{f^k}, x\rangle y^{\underline{k}} - \langle g, x\rangle y^g)\rangle \\
& = \langle f^j, x \rangle - \sum_{k=1}^{\infty}\langle f^k, x\rangle  \langle f^j, y^{\underline{k}}\rangle - 
\langle g, x\rangle \langle \widehat{f^j}, Ly^g \rangle\\
& = \langle f^j, y^g\rangle = 0.
\end{aligned}
\]
Moreover,
\[
\begin{aligned}
\langle g, L \circ \widetilde{Q}x\rangle & = \langle g, L(x-\sum_{k=1}^{\infty}\langle \widehat{f^k}, x\rangle y^{\underline{k}} - \langle \widehat{g}, x\rangle y^g)\rangle\\
& = \langle g, x \rangle - \sum_{k=1}^{\infty}\langle f^k, x \rangle \langle g, y^{\underline{k}}\rangle - \langle g, x\rangle \langle g, Ly^g \rangle\\
& = \langle \widehat{g}, x\rangle - \langle \widehat{g}, x \rangle \langle \widehat{g}, y^g\rangle =0.
\end{aligned}
\]
Consequently, the composite map $ L \circ \widetilde{Q}|_{X}$ takes values in $Y$. Moreover, for $\widetilde{Q}y =y\in Y$  and $ (L \circ \widetilde{Q})y =y$ ($y \in Y$), which demonstrates that $\widetilde{U} \in \mathcal{P}(X,Y). $ Since $ \|L\|=1,$
\[
    \lambda(\widehat{Y}, X^{**}) \geq \lambda(Y,X).
\]

For the converse, we fix a sequence $(P_k)_{k=1}^\infty$ in  $ \mathcal{P}(X,Y)$ such that $ \|P_k \| \rightarrow \lambda(Y,X)$ as $k\to\infty$. Let $(H_k)_{k=1}^\infty$ be as in Lemma~\ref{thm6}. We then define 
\[
    U_k =\frac{1}{(n!)^k}\sum_{h \in H_k} (h^{-1}\circ P_k \circ h).
\]
We have $ U_k \in \mathcal{P}(X,Y), $ $\|U_k\| \leq \|P_k\|.$ Let $(y^{(g)})_{k=1}^\infty$ associated to $ U_k$ in accordance with Lemma~\ref{lem0} ($(k\in \mathbb N)$). We then set
\[
    Q_k = U_k + \langle g, \cdot\rangle y^{(g)}_k\quad (k\in \mathbb N).
\]
We \emph{claim} that $ Q_k \in \mathcal{P}(X,Y_1).$ For this, notice that for any $ x \in X$ and $ j \in \mathbb{N}$ we have
\[
    \langle f^j, Q^kx \rangle = \langle f^j, P_kx\rangle + \langle g, x\rangle \langle f^j, y^{(g)}_k\rangle = 0. 
\]
If $ x \in Y_1,$ then $P_k|_{Y_1} \in \mathcal{P}(Y_1,Y)$, which implies that  
\[
    Q_kx = P_kx + \langle g, x\rangle y^{(g)}_k = x - \langle g, x \rangle y^{(g)}_k + \langle g, x\rangle y^{(g)}_k = x
\]
and proves the claim. \smallskip

Reasoning as in Theorem~\ref{thm6}, we conclude that $ Q_k \rightarrow Q \in \mathcal{P}(X, Y_1)$ in the topology of
pointwise-to-weak$^*$ convergence in $ X^{**}.$ Let $ (y_j)_{j=1}^\infty $ be the sequence corresponding to $Q$ in accordance with Lemma~\ref{lem0}.  By the same argument as in the proof of Proposition~\ref{thm6}, we infer that $ y_j = y^{\underline{j}} $ ($j \in \mathbb{N}$)
where $ y^{\underline{j}}$ are defined as in the proof of Lemma~\ref{lemma:2}. Consequently,
\[
    Qx = x - \sum_{j=1}^{\infty} \langle f^j, x\rangle  y^{\underline{j}}.
\]
We \emph{claim} that the sequence $ (y^{(g)}_k)_{k=1}^\infty $ is bounded. Assume not and fix a norm-one vector $ x \in Y_1 \setminus Y$. Then 
\[
    \|U_k\| \geq \|U_k x\| = \|x -\langle g, x\rangle y^{(g)}_k\| \geq |\langle g, x\rangle| \|y^{(g)}_k\| - 1 \rightarrow \infty;
\]
a contradiction. We may then pick a weak$^*$ cluster point $ y^g \in X^{**}$ of $ (y^{(g)}_k)_{k=1}^\infty$. Consequently, $ (U_k)_{k=1}^\infty$ has a subnet convergent to $Q - \langle g, \cdot \rangle y^g $ pointwise-to-weakly$^*$ in $X^{**}.$ 
Let $R\colon X^*\to X^*$, $\widetilde{P}\colon X^{**} \rightarrow X^{**}$ be defined by
\[
    \widetilde{P}x  = x - \sum_{j=1}^{\infty}  \langle x, \widehat{f^j}\rangle y^{\underline{j}}-\langle x, \widehat{g}\rangle y^{g}\quad (x\in X^{**}).
\]

\[
    Rf  = f - \sum_{j=1}^{\infty}  \langle f, y^{\underline{j}} \rangle {f^j}  - \langle y^{g}, f\rangle  {g} \quad (f \in X^*).
\]

Then, $R^* = \widetilde{P}$, so $\widetilde{P}$ is weak*-to-weak* continuous.
We \emph{claim} that $ \widetilde{P}(X^{**}) \subseteq \widehat{Y}.$ For this, take $ x \in X^{**}$.
\begin{itemize}
    \item When $x\in \widehat{Y} \oplus W \subset X$, where $W$ is as in Lemma~\ref{lem0}, we have $x\in \widehat{Y} \oplus W_n$ for some $n\in \mathbb N$. Since $ \widehat{Y} \oplus W_n$ is contained in $\bigcap_{j=n+1}^{\infty} \ker \widehat{f^j}$, we have
    \begin{equation}\label{eq:trun}
        \widetilde{P} x = x - \sum_{j=1}^{n} \langle \widehat{f^j}, x\rangle y^{\underline{j}}-\langle \widehat{g}, x\rangle y^{g}.
    \end{equation}
    Since $\langle g, y^{\underline{j}}\rangle = 0 $  and $\langle f^j, y^g\rangle =0$ ($j=1, \ldots, n$), $\widetilde{P}x \in \widehat{Y}$. 
    \item In the general case where $x\in X^{**}$, we may use Goldstine's theorem to find a net $(x_\alpha)$ in $\widehat{Y} \oplus W$ that converges weak* to $x$. Since $\widetilde{P}$ is weak*-to-weak* continuous and $\widehat{Y}$ is weak*-closed being the intersection of kernels of weak*-continuous functionals, $\widetilde{P}x_\alpha \to \widetilde{P}x \in \widehat{Y}$. \end{itemize}
Thus, for a unit vector $ x \in \widehat{Y}$, $ \widetilde{P}x =x$ and for some $n$ (\emph{cf}. \eqref{eq:trun}) we have
\[
    \|\widetilde{P} x \| = \|x - \sum_{j=1}^{n}\langle \widehat{f^j}, x\rangle  y^{\underline{j}}-\langle \widehat{g}, x\rangle y^{g} \| \leq \liminf_k \|U_k\|.
\]
Since $\widehat{Y} \oplus W$  is weak* dense in $X^{**}$ we arrive at the the estimate
\[
    \lambda(\widehat{Y}, X^{**}) \leq \|\widetilde{P}\| \leq \liminf_k \|U_k\| = \lambda(Y,X), 
\]
as required.

By the first part of the proof, $\widetilde{U} = L \circ \widetilde{Q}|_{X}$ is a minimal projection in $\mathcal{P}(X,Y).$ Notice that 
\[
    \widetilde{U} x = x - \sum_{j=1}^{\infty} \langle f^j, x \rangle y^{\underline{j}}-\langle g, x\rangle y^{g}\quad (x\in X).
\]
By Lemma~\ref{lemma:2}, $\| \widetilde{U}\| = \lambda(Z_n, X) + \|y^g\|.$ By Lemma~\ref{lem:summing}, $\lambda(Z_n, X) = 2 - \frac{2}{n}.$ Let us fix $ y \in X$ such that  $\langle g, y\rangle =1 $ and $ \langle f^j, y \rangle =0$ ($ j \in \mathbb{N}$). We define
\[
    P_y x = x - \sum_{j=1}^{\infty} \langle f^j, x\rangle y^{\underline{j}}-\langle g, x\rangle y\quad (x\in X).
\]
Reasoning as above, we infer that $P \in \mathcal{P}(X,Y).$ Moreover, by Lemma~\ref{lemma:2} applied to $P_y$, we conclude that $\|P_y\| = \lambda(Z_n, X) + \|y\|$. 
On the other hand, $ \|P_y\| \geq \|\widetilde{U}\| = \lambda(Y,X)$. Consequently, $ \|y\| \geq \|y^g\|$. 
Hence
\[
    \|y^g\| = \inf \{ \|y\|\colon y \in X, \langle g, y\rangle =1, \langle f^j, y\rangle = 0 \; (j\in \mathbb{N}) \}
\]
and
\[ 
    \lambda(Y,X)=2-\frac{2}{n} + \frac{1}{\|g|_{Y_1}\|},
\]
as required.
\end{proof}
Theorem~\ref{Thm:B} then follows from Proposition~\ref{prop:B:final}.
}
\subsection*{Acknowledgements} We are greatly indebted to the referee for careful reading of the manuscript and detecting a number of slips whose removal improved greatly the presentation.

\bibliographystyle{plain}

\end{document}